\documentclass[microtype]{gtpart}
\title{Surjective morphisms from affine space to its Zariski open subsets}

\author{Viktor Balch Barth}
\email{viktorbb@math.uio.no}
\address{Department of Mathematics, UiO, Norway.}

\usepackage{hyperref}
\usepackage{setspace}
\usepackage{graphicx}
\usepackage{amssymb}
\usepackage{mathrsfs}
\usepackage{amsthm}
\usepackage{amsmath}
\usepackage{color}
\usepackage[Lenny]{fncychap}
\usepackage[font=small,labelfont=bf]{caption}
\usepackage{fancyhdr}
\usepackage{times}
\usepackage[capitalize]{cleveref}
\usepackage{float}
\usepackage{extarrows}
\restylefloat{figure}
\usepackage{algorithm}
\usepackage{algpseudocode}

\usepackage{mathtools}
\usepackage{faktor}
\usepackage{amsfonts}
\usepackage{color}
\usepackage{graphicx}
\usepackage{xspace}
\usepackage{hyperref}
\usepackage{blindtext}
\usepackage{url}
\usepackage{listings}
\usepackage{graphicx}
\usepackage{float}
\usepackage{xcolor}
\usepackage{etoolbox}
\usepackage{caption}
\usepackage{xpatch}
\usepackage{appendix}
\usepackage{parskip}
\usepackage{tikz}
\usepackage{tikz-cd}
\usepackage{enumerate}
\usepackage{pgfplots}
\pgfplotsset{compat=newest}
\usepackage{wrapfig}
\usepackage{stmaryrd}
\usepackage{comment}
\usepackage{tikz}
\usepackage{pgfplots}
\usepackage{tikz-3dplot}
\tdplotsetmaincoords{60}{115}
\pgfplotsset{compat=newest}
\usetikzlibrary{calc}
\usepackage{subcaption}
\usepackage{comment}

\usepackage{datetime}

\newdateformat{monthyeardate}{%
  \monthname[\THEMONTH], \THEYEAR}

\pgfplotsset{compat=1.15}

\xspaceaddexceptions{]\}}
\newcommand{\fix}[1]{\ifmmode{#1}\else{$#1$}\xspace\fi}
\newcommand{\A}{\mathbb{A}}
\newcommand{\an}{\mathbb{A}^{n}}
\newcommand{\ak}{\mathbb{A}^{k}}

\newcommand{\inv}{^{-1}}

\newcommand{\at}{\mathbb{A}^2}

\renewcommand{\phi}{\varphi}
\renewcommand{\epsilon}{\varepsilon}

\renewcommand{\and}{\quad \text{ and } \quad}

\newcommand*{\doublerightarrow}[2]{\mathrel{
  \settowidth{\@tempdima}{$\scriptstyle#1$}
  \settowidth{\@tempdimb}{$\scriptstyle#2$}
  \ifdim\@tempdimb>\@tempdima \@tempdima=\@tempdimb\fi
  \mathop{\vcenter{
    \offinterlineskip\ialign{\hbox to\dimexpr\@tempdima+1em{##}\cr
    \rightarrowfill\cr\noalign{\kern.5ex}
    \rightarrowfill\cr}}}\limits^{\!#1}_{\!#2}}}

\newcommand\restr[2]{{
  \left.\kern-\nulldelimiterspace 
  #1 
  \vphantom{\big|} 
  \right|_{#2} 
  }}

\theoremstyle{plain}
\newtheorem{thm}{Theorem}[section]
\newtheorem{cor}[thm]{Corollary}

\newtheorem{lem}[thm]{Lemma}
\newtheorem{prop}[thm]{Proposition}
\theoremstyle{definition}

\newtheorem{exam}[thm]{Example}

\newtheorem{defn}[thm]{Definition}

\begin{document}

\begin{abstract}  
We prove constructively the existence of surjective morphisms from  affine space onto certain open subvarieties of affine space of the same dimension. 
For any algebraic set $Z\subset \mathbb{A}^{n-2}\subset \mathbb{A}^{n}$, we construct an endomorphism of $\mathbb{A}^{n}$ with $\mathbb{A}^{n} \setminus Z$ as its image. 
By Noether's normalization lemma, these results extend to give surjective maps from any $n$-dimensional affine variety $X$ to $\mathbb{A}^{n} \setminus Z$.
\end{abstract}

\maketitle

\section{Introduction}
In the context of Oka theory, Forstneri\v{c} \cite{Forstneric2017surjective} showed that every connected Oka manifold $Y$ admits a surjective holomorphic map from an affine space $\mathbb{A}^N$. 
Moreover, this affine space can be taken to be $\mathbb{A}^{\dim Y}$. 
Motivated by this, Arzhantsev and Kusakabe \cite{arzhantsev2022images,kusakabe2022surjective} recently proved analogous results in the algebraic setting. 
For this paper, we consider algebraic varieties over an algebraically closed field $\mathbb{K}$ of characteristic zero.
\begin{defn}
    An algebraic variety $Y$ is called an \emph{$A$-image} if for some positive integer $N$ there is a surjective morphism $f: \mathbb{A}^N \to Y$.
\end{defn}
Arzhantsev proved that every very flexible variety is an $A$-image \cite{arzhantsev2022images}.
As shown in \cite{arzhantsev2013flexible}, every very flexible variety is smooth and subelliptic.
Kusakabe extended Arzhantsev's result by showing that any smooth subelliptic variety $Y$ admits a surjective morphism $f: \mathbb{A}^{\dim Y+1} \to Y$ \cite{kusakabe2022surjective}. 
For a definition of algebraic subellipticity, see \cite[Definition 5.6.13 (e)]{Forstneric2017book}, and note that it was recently shown to be equivalent to algebraic ellipticity for smooth varieties \cite{kaliman2023gromov}.
Whether any $A$-image $Y$ also admits surjective morphisms $f: \mathbb{A}^{\dim Y} \to Y$ is an open question 
\cite{arzhantsev2022images,Larusson2017Approx,Forstneric2017surjective,FORSTNERIC2023}. 
When $Y$ is a smooth \emph{proper} subelliptic variety, Forstneri\v{c} proved over the complex numbers that such a surjective map from $\mathbb{A}^{\dim Y}$ does indeed exist \cite[Theorem 1.6]{Forstneric2017surjective}. It is worth pointing out that Forstneri\v{c} and Kusakabe's results give maps which are surjective even when restricted to the smooth locus $\an \setminus \left(\mathrm{Sing}(f)\right)$. 

In this paper, we ask whether there are surjective morphisms $f: \mathbb{A}^{\dim Y} \to Y$ for open subvarieties $Y \subset \an$. For certain families of cases we give a positive answer. First of all, note that any $A$-image $Y$ must have complement $\an \setminus Y$ of codimension at least two. If not, then $Y$ admits a nonconstant invertible regular function $g$ which pulls back to a nonconstant invertible regular function $f^*(g)$ on $\mathbb{A}^N$, giving a contradiction. 
However, as long as $\an\setminus Y$ has codimension at least two, $Y$ is an $A$-image  \cite{arzhantsev2022images,kusakabe2022surjective}. 

In the two-dimensional case, Jelonek gave the example of a surjective map $f: \at \to \at \setminus \{(0,-1)\}$ defined by $f(z_1,z_2) = (z_1(z_1z_2+1)-z_2,z_1z_2)$  \cite{Jelonek1999Anumberof}. He also proved that no such map of lower degree exists. Modifying and generalizing this example yields the constructions used in \cref{thm: surjective onto complements of affine variety} and \cref{thm: surjective onto complement of points}. Note that Jelonek's map answered the question of the existence of a surjective morphism $f: \at \to \at \setminus \{(0,0)\}$ even before it was asked in \cite{Larusson2017Approx} (and then answered again in \cite{arzhantsev2022images}). 

In what follows we state our main results and their corollaries, postponing the proofs until \cref{sec: proofs}. 

\begin{thm}\label{thm: surjective onto complements of affine variety}
    Let $Z\subset \an$ be an algebraic set of the form $Z = F 
    \times W \subset \mathbb{A}^{2} \times \mathbb{A}^{n-2}$, where $F\subset \mathbb{A}^{2}$ is a finite set of $l$ points, and $W\subset \mathbb{A}^{n-2}$ is the zero set of $m$ polynomials $q_i$ of degree at most $d$.
    Then there exists a surjective map
    $f: \mathbb{A}^{n} \to \mathbb{A}^{n} \setminus Z $ of degree $\deg(f) \leq \max(1, l-1)\cdot 
    \max(l+2, m+d+1)$.   
\end{thm}
Given a closed algebraic set $Z\subset \ak$ satisfying some mild conditions, there are bounds on the number $m$ of polynomials required to specify $Z$ as a set, and on the maximal degree $d$ of these polynomials. For $Z$ an irreducible closed subset of $\ak$, \cite[Proposition 3]{Heintz1983Defin} states that it is described by $k+1$ polynomials of degree bounded by $\deg Z$. 
If $Z$ is a smooth equidimensional variety, then by \cite[Theorem I]{blanco2004computing} it is described by $m= (k-\dim Z) (1+\dim Z)$ polynomials, of degree bounded by $\deg Z$. As a consequence, we have the following corollary.
\begin{cor}
If $Z\subset \A^{n-2}\subset \A^{n}$ is an irreducible closed 
algebraic set, 
then there exists a surjective map $f: \an \to \an\setminus  Z$ with $\deg(f) \leq n + \deg(Z)$. If $Z\subset \A^{n-2}\subset \A^{n}$ is a smooth equidimensional 
variety,  
then there exists a surjective map $f: \an \to \an\setminus Z$ with $\deg(f) \leq (n-2-\dim(Z))(1+\dim(Z))+ \deg(Z)+1$.\\
\end{cor}

Srinivas and Kaliman \cite{srinivas1991embedding,kaliman1991extensions} showed that if  $n\geq \max(2\dim Z_1+1, \dim TZ_1)+1$, and $\phi: Z_1 \to Z_2$ is an isomorphism of closed varieties in $\an$, then $\phi$ extends to an $\an$-isomorphism. We obtain the following corollary.
\begin{cor}
    Let $Z_1 \subset \mathbb{A}^{n-2}\subset \an$ be a closed variety such that $n\geq \max(2\dim Z_1+1, \dim TZ_1)+1$. Then for any $Z_2\subset \an$ isomorphic to $Z_1$, there exists a surjective map $f: \an \to \an \setminus Z_2$.
\end{cor}

Over the real numbers, Fernando and Gamboa \cite{FernandoGamboa} showed the existence of a surjective morphism $f: \mathbb{R}^n \to \mathbb{R}^n\setminus F$ 
for any finite set $F$ of $l$ points, and $n\geq2$. For $F$ a subset of a line, their map $f$ has degree $2(l+2)$. Morphisms of lower degree, namely $\deg(f) = l+2$, were obtained by El Hilany \cite{el2022counting} in the setting where $F$ is a subset of a line, $l$ is even, in dimension $2$ over $\mathbb{K}$ an algebraically closed field of characteristic zero. 
In the following proposition, by use of a different construction, we extend El Hilany's degree bound to also hold for $l$ odd, and for any $n\geq 2$.

\begin{prop}\label{thm: surjective onto complement of points}
For $n\geq 2$, and any set of $l$ points $F 
\subset \an$, there exists a morphism $f: \an \to \an$ such that $f(\an)= \an\setminus F$. 
In general, if $l>2$, then there exists an $f$ with $\deg(f) \leq (l-1)\cdot (l+2)$. If $F$ is a subset of a line, then there exists $f$ with $\deg(f) = l+2$.  
\end{prop}

In particular, if $l < 3$, then there is a map $f$ with $\deg (f) = l+2$. It is worth noting that for $Y$ open in $\at$, $Y$ is an $A$-image precisely if $\at \setminus Y$ has codimension two. Hence any two-dimensional $A$-image $Y\subset \at$ admits a surjective morphism from $\at$. 

Finally, recall that the Noether normalization lemma \cite[Theorem 13.3]{eisenbud2013commutative} states that for any $n$-dimensional affine variety $X$, there exists a finite, hence surjective, morphism $X \to \an$. As a composition of surjections is surjective, we have the following proposition.

\begin{prop}\label{prop: Noether}
    For any $n$-dimensional affine variety $X$, and any $Y$ admitting a surjective map $f: \an \to Y$, there exists a surjective map
    $$\phi: X \to \an \to Y.$$
\end{prop}
In particular, this shows that from any $n$-dimensional affine variety $X$, there are surjective maps onto varieties $Y$ as in \cref{thm: surjective onto complements of affine variety} and \cref{thm: surjective onto complement of points}. By combining Noether normalization with Kusakabe's result \cite{kusakabe2022surjective}, we conclude that any smooth subelliptic variety $Y$ of dimension $n$ admits a surjective map from any $(n+1)$-dimensional affine variety $X$.

\section*{Acknowledgements}
The author is supported by Young Research Talents grant 300814 from the Research Council of Norway. The author is very grateful to Tuyen Trung Truong for introducing him to the problem and for providing helpful guidance. The author would also like to thank the reviewer, as well as Ivan Arzhantsev, Boulos El Hilany, Franc Forstneri\v{c}, Zbigniew Jelonek, Shulim Kaliman, Yuta Kusakabe, and Mikhail Zaidenberg for helpful comments and clarifications. 

\section{Proofs and explicit constructions}\label{sec: proofs}
We first construct a map $g$ in \cref{lem: surjective onto OOW}, which we then use to prove \cref{thm: surjective onto complements of affine variety}. We then prove \cref{thm: surjective onto complement of points}.

\begin{lem}\label{lem: surjective onto OOW}
Let $g: \an_z \to \an_w$ be given by 
\begin{equation}
    g=\big(z_{1}+z_{2}(r(z_{1})z_{2}+1) 
    +\sum_{i=1}^{m} z_2^{i+1}q_i ,\, 
    r(z_{1})z_{2}+1 ,\, 
    z_3 ,\, \ldots ,\, 
    z_n \big), \label{eq: function f}
\end{equation}
where $r(z_1)$ has $l$ simple roots $\{\beta_j\}_{j}$, and the $m$ polynomials $q_i(z_3,z_4,\ldots,z_n)$ have degree at most $d$. Then $g$
has image $\an \setminus Z$, where $Z = \{\beta_j\}_{j}  \times \{0\} \times W$, where 
$W\subset \mathbb{A}^{n-2}$ is the set of common zeros of the $q_i$. 
\end{lem}
\begin{proof}
Fix $w\in \an_w$, and notice that any point $z$ in the preimage $g\inv(w)$ must satisfy $z_3=w_3,\ldots, z_n=w_n$. 
Fixing the values of $z_3, \ldots, z_n$  will in turn determine the values $q_i(z_3,\ldots, z_n) = \alpha_i$.
Thus whether $g\inv(w)$ is empty only depends on the values of $w_1,w_2,\alpha_1,\ldots,\alpha_{m}$, and it reduces to solving 
the following system of equations for $z_{ 1 }$ and $z_{ 2 }$.
\begin{align}
    w_1 &= z_{ 1 } + z_{ 2 }(r(z_{ 1 })z_{ 2 }+1) +\sum_{i=1}^{m} z_{ 2 }^{i+1}\alpha_i \label{eq: w1 equals stuff} \\
    w_2 &= r(z_{ 1 })z_{ 2 }+1. \label{eq: w2 equals stuff}
\end{align}
Solve \cref{eq: w1 equals stuff} for $z_{ 1 }$, then substitute for $z_1$ in \cref{eq: w2 equals stuff} and simplify to obtain
\begin{align}
    z_{ 1 } &= w_1 -z_{ 2 }w_2 -\sum_{i=1}^{m} \alpha_i z_{ 2 }^{i+1} \label{eq: general firsteq}  \\
    0 &= 1-w_2 +z_{ 2 } \cdot r\left(w_1 -w_2z_{ 2 } - \sum_{i=1}^{m} \alpha_i z_{ 2 }^{i+1}\right). \label{eq: general secondeq}
\end{align}
If $w\in \left(\{\beta_j\}_j \times \{0\}\times W\right)$, then $(w_1,w_2,\alpha_1, \ldots, \alpha_{m}) = (\beta_j,0,\ldots,0)$ for some $j$, and \cref{eq: general secondeq} reduces to $0 = 1$, which is clearly unsolvable.

If $w\notin \left(\{\beta_j\}_j \times \{0\}\times W\right)$, then $(w_1,w_2,\alpha_1, \ldots, \alpha_{m}) \neq (\beta_j,0,\ldots,0)$ for any $j$, and one may solve \cref{eq: general secondeq} for $z_{ 2 }$, and then solve \cref{eq: general firsteq} for $z_{ 1 }$. Hence, $g(\an)= \an\setminus Z$.
\end{proof}

We may now prove the main theorem.
\begin{proof}[Proof of \cref{thm: surjective onto complements of affine variety}]
Given $Z= F \times W$ as in the theorem statement, we modify the function $g$ in \cref{eq: function f} by composing with suitable automorphisms. 
For any $F\subset \at$, first pick a convenient coordinate system by applying an affine $\at$-automorphism $t$ such that $t(F) = \{p_j\}_j$, where the points $p_j$ have pairwise distinct first coordinates $\beta_j = t(p_j)^{(1)}$. To simplify the exposition, we assume $F= \{p_j\}_j$.
Define 
\begin{equation}
r(z_1)= \prod_{j=1}^m \left(z_1- \beta_j\right),
\end{equation}
and use \cref{lem: surjective onto OOW} to construct a surjective map $g: \an \to \an \setminus \left( \{\beta_j\}_{j}  \times \{0\} \times W \right)$. 
Let $L(z_1)$ be a Lagrange polynomial such that $L(\beta_j) = p_j^{(2)}$ for all $j$, and $h=(z_1, z_2+L(z_1), z_3, \ldots, z_n)$, and set $f= h \circ g$. Then $f\left(\an\right) = \an \setminus Z$. 
The degree of $h$ is at most $\max(1,l-1)$, and $g$ has degree at most $\max(l+2, m+d+1)$. Hence the degree of $f$ is bounded by $\max(1, l-1)\cdot \max(l+2, m+d+1)$.
\end{proof}

\begin{exam}
    Let $Z\subset \A^4$ be two parallel copies (contained in $\{(1,0)\}\times \at$ and $\{(-1,0)\}\times \at$ respectively) of a nodal cubic curve $W=V(w_3^2-w_4^3-w_4^2)\subset \at$. Then 
    \begin{equation}
        f=\big(z_1+ z_2((z_1^2-1)z_2+1)+z_2^2(z_3^2-z_4^3-z_4^2),\, (z_1^2-1)z_2+1 ,\, z_3 ,\, z_4 \big)
    \end{equation}
    is a surjective map $f:\A^4 \to \A^4\setminus Z$.
\end{exam} 

\begin{exam} \label{ex:many points}
    Let $d\geq n>2$ be natural numbers, and for $i = 1,\ldots, n-2$, let $q_i(z_{i+2})= \prod_{j=1}^{j=d-i-1} (z_{i+2}-j)$. Similarly, let $r(z_1)= \prod_{j=1}^{j=d-2} (z_{1}-j)$. Then $f: \an \to \an$ given by 
    \begin{equation}
        f=\big(z_{ 1 } + z_{ 2 }(r(z_{ 1 })z_{ 2 }+1) 
    +\sum_{i=1}^{n-2} z_{ 2 }^{i+1}q_i ,\, 
    r(z_{ 1 })z_{ 2 }+1 ,\, 
    z_3 ,\, \ldots ,\, 
    z_n \big),
    \end{equation}
    has degree $d$, and the complement of the image of $f$ is the set $\{1,2,\ldots, d-2\}\times\{0\}\times\{1,\ldots, d-2\}\times \{1,\ldots, d-3\} \times \ldots \times \{1,2,\ldots, d-n+1\}$, consisting of $(d-2)\cdot (d-2)!/(d-n)!$ points.
\end{exam}
The above example gives a function with superexponentially many points in the complement of the image as $n$ and $d$ grow. We now construct $f$ in \cref{thm: surjective onto complement of points}, which only avoids $\deg(f)-2$ points. However, for any $n$, the proposition gives a surjective map $f:\an \to \an \setminus \{0\}$ of degree only $\deg(f)=3$.

\begin{proof}[Proof of \cref{thm: surjective onto complement of points}]
Given $l$ points $F= \{p_j\}_j \subset \an$, pick a coordinate system in which $p_j^{(1)}$ are all pairwise distinct, as in the proof of \cref{thm: surjective onto complements of affine variety}. Define the polynomial 
$r(z_1)= \prod_{j=1}^l \left(z_1- p_j^{(1)}\right)$.
Then define $\sigma_{2,r}:\at_z \to \at_w$ to be
\begin{equation}
    \sigma_{2,r}=\big(z_{1} +z_{2}(r(z_{1})z_{2}+1) ,\, 
    r(z_{1})z_{2}+1\big).
\end{equation}
Note that $\sigma_{2,r}$ is equal to $g$ defined in \cref{eq: function f}, and hence $\deg (\sigma_{2,r})=l+2$. 
For $n>2$, define $\sigma_{n,r}:\an_z \to \an_w$ to be
\begin{equation}
    \sigma_{n,r}=\big(z_{1} +z_{2}(r(z_{1})z_{2}+1) ,\, 
    r(z_{1})z_{2}+1+z_3,\, z_3^2+z_4,\, \ldots,\, z_{n-1}^2+z_n,\, z_n^2\big).
\end{equation} 
Let $n>2$ and fix $w\in \an_w$. The same calculation as in the proof of \cref{lem: surjective onto OOW} leads to an equation similar to \cref{eq: general secondeq}, namely
\begin{align}
    0 &= 1-(w_2-z_3) +z_{ 2 } \cdot r\left(w_1 -(w_2-z_3)z_{ 2 }\right). \label{eq: general secondeq for points}
\end{align}
This is unsolvable for $z_2$ if and only if $w_1 \in \{p_j^{(1)}\}_j$ and $w_2=z_3$. However, $z_3 = \pm \sqrt{w_3 \mp \sqrt{w_4 \pm \sqrt{\ldots}}}$, which has multiple solutions unless $w_3=0,\ldots,w_n=0$. 
Hence $\sigma_{n,r}(\an) = \an \setminus \left(\{p_1^{(1)},p_2^{(1)},\ldots,p_l^{(1)}\}\times \{0\}\times \ldots \times \{0\}\right)$. 
Similarly to the proof of \cref{thm: surjective onto complements of affine variety}, we may construct an $\an$-automorphism $h$ by use of interpolation polynomials $L_i(z_1)$ such that $L_i\left(p_j^{(1)}\right) = p_j^{(i)}$ for all $j$. By composing, we obtain $f=h\circ\sigma_{n,r}$, which satisfies $f(\an) = \an \setminus F$, and has degree bounded by $\deg(f) \leq (l-1)(l+2).$ If $F$ is a subset of a line, then $h$ can be an affine automorphism, and we obtain $\deg(f)=l+2$.
\end{proof}

\bibliography{mylib}
\end{document}